\documentclass[11pt]{amsart}
\usepackage {amsmath, amssymb, a4wide, epsfig, enumerate, psfrag}
\usepackage[latin1]{inputenc}
\usepackage[english]{babel}

\date{\today}

 \author{Romain Dujardin}

\address{CMLS \\ \'Ecole Polytechnique \\ 91128 Palaiseau\\
         France}
\email{dujardin@math.polytechnique.fr}
\thanks{Not intended for publication}

\title{A note on the rank of positive closed currents}

\newcommand{\cc}{\mathbb{C}}

\newcommand{\pp}{\mathbb{P}}
\newcommand{\e}{\varepsilon}
\newcommand{\cv}{\rightarrow}

\newcommand{\om}{\Omega}

\newcommand{\rest}[1]{ \arrowvert_{#1}}

\newcommand{\unsur}[1]{\frac{1}{#1}}

\newcommand{\bra}[1]{\left\langle #1\right\rangle}

\DeclareMathOperator{\rank}{rank}

\newtheorem{prop} {Proposition} 
\newtheorem{thm}[prop] {Theorem} 

\newtheorem{lem}[prop] {Lemma}

\theoremstyle{remark}

\newtheorem{rmk}[prop]{Remark}

\begin{document}
 
\maketitle

The purpose of this note is to prove the following theorem, which was stated without   proof in \cite{fatou}.

\begin{thm}\label{thm:proj}
Let $T$ be a strongly positive closed current of bidimension $(p,p)$ on $\mathbb{P}^k$, and assume that its trace measure has dimension 
$\dim(\sigma_T)< 4p$.
Then $\sigma_T$-a.e. we have that
\begin{equation}\label{eq:dim}
p\leq \mathrm{rank} (T) \leq \unsur{2}\dim (\sigma_T).
\end{equation}
\end{thm}

We refer to \cite{lelong, demailly} or \cite[\S 1,2]{fatou} for detailed basics on positive exterior algebra and positive currents. Here we just recall that a positive current $T$ of bidimension $(p,p)$ admits an \textit{integral representation} in the sense that there exists a measurable field $t_T$ of positive $(p,p)$ vectors such that 
 for any test $(p,p)$-form $\varphi$
$$\langle T,\varphi \rangle = \int \langle t_T, \varphi \rangle \sigma_T, $$ where $\sigma_T$ is the trace measure. If $t_T(x)$ is well-defined (which happens $\sigma_T$-a.e.), the rank of $T$ at $x$ is the rank of $t_T$, that is the dimension of the smallest sub-vector space $W$ of $T_x\mathbb{P}^k$ such that  $t_T(x)\in\bigwedge^{p,p}(W)$. A positive 
$(p,p)$ vector is decomposable if and only if its rank equals $p$. 

\medskip

Before starting the proof, we also recall the following  result  
  \cite[Corollary 2.5]{fatou}.

\begin{thm} \label{thm:rank}
Let $T$  
be a strongly positive current of bidegree $(q,q)$ in $\om\subset \cc^k$. Assume that  the family  $(T_\e^2)_{\e>0}$ has locally uniformly bounded mass as $\e\cv 0$.

Then if $T^2_{\rm ac}=0$,  $T$ has rank $<k$ a.e.
\end{thm}

We refer to \cite{fatou} for the precise definition of $T^2_{\rm ac}$ for a positive current $T$.

\begin{proof}[Proof of Theorem \ref{thm:proj}]
By definition $\rank(T)\geq p$ so only the  inequality  $\mathrm{rank} (T) \leq \unsur{2}\dim (\sigma_T)$ needs to be established. Let $\ell = \lfloor \unsur{2}\dim(\sigma_T)\rfloor +1$. Let $q=k-p$.

If $I$ (resp. $L$) is a linear subspace of dimension $k-\ell-1$ (resp. $\ell$), such that $I\cap L=\emptyset$,  we can consider the linear projection of center $I$, $\pi_I: \pp^k\setminus I\cv L$. If $I$ is fixed, changing $L$ amounts to post-composing $\pi_I$ with a linear automorphism, so we may
simply think of $\pi_I$ as  mapping $\pp^k\setminus I$ onto $\pp^\ell$. 

For generic $I$ the projection $(\pi_I)_*T$ is a well-defined  positive current of bidimension $(p,p)$ in $L\simeq \pp^\ell$, in the sense that it satisfies the property $\langle(\pi_I)_*T, \varphi\rangle=
\bra{T, \pi_I^*\varphi}$  for every test $(p,p)$ form, and it has the same mass as $T$.
Indeed for this it is enough to resove $\pi_I$ by writing it as $\beta\circ\alpha^{-1}$, where $\alpha$ and $\beta$ are holomorphic, and define $(\pi_I)_* = \beta_*\alpha^*$. The operator $\alpha^*$ is always well-defined  on  compact K\"ahler manifolds,
even if it is not always continuous (see \cite{ds-pullback} for details).
Shortly we'll see that $(\pi_I)_*T$ is strongly positive.
Notice that $\ell-p\leq \ell/2$, as follows from our assumption on $\dim (\sigma_T)$.

\medskip

So from now on we consider $I$ such that $(\pi_I)_*T$ is  well-defined and $\sigma_T(I)=0$, and we simply write $\pi$ for $\pi_I$. We also denote by $\omega_L$ the restriction of $\omega$ to $L$.
Fix  a Borel set $E$ such that $\sigma_T(E) = 1$ and $t_T(y)$ exists at every $y\in E$.

The first claim is that $\sigma_{\pi_*T}\ll\pi_*\sigma_T$. Indeed observe first that
$\pi^*(\omega_L^p)\ll \omega^p$, thus $T\wedge\pi^*(\omega_L^p)\ll T\wedge \omega^p$. Next, we have the formulas
$\sigma_{\pi_*T} = \pi_*\left( T\llcorner\pi^*\omega_L^p \right)$ and
$\pi_*\sigma_T = \pi_*\left( T\llcorner\omega^p\right) $ and the result easily follows.

From this we deduce that $\dim((\pi_I)_*T)\leq\dim(\sigma_T)<2\ell$.
Indeed, since $\pi$ is locally Lipschitz outside $I$, $\mathrm{HD}(\pi_I(E))\leq \mathrm{HD}(E)$, and
$\pi_I(E)$ is a set of full mass for $\sigma_{(\pi_I)_*T}$.

Conversely for generic $I$, $\pi_*\sigma_T\ll\sigma_{\pi_*T}$.
Indeed, if not,  there is a set $A$ of positive $\sigma_T$ mass  such that if $x\in A$
$$t_T(x)\llcorner\pi^*(\omega_L^p) = \langle t_T(x),\pi^*(\omega_L^p)\rangle
=\langle \pi_*(t_T(x)),\omega_L^p\rangle =0,$$ thus  $\pi_*(t_T(x))=0$. This means that the decomposable vectors making up $t_T(x)$ are not in general position with respect to the fibers of $\pi_L$. More precisely, if $t$ is such a vector, $\mathrm{Span}(t)$ will not be transverse to the fiber, which has dimension $k-\ell<k-p$. This can only  happen for a set of projections of zero measure (see  Lemma \ref{lem:linalg} below). We conclude that the existence of such a $A$ is not possible for generic $I$.
From now on we assume that $I$ is chosen so that $\pi_*\sigma_T\ll\sigma_{\pi_*T}$, and we
let $h\in L^1_{\rm loc}(\sigma_{\pi_*T})$ such that $\pi_*\sigma_T = h \sigma_{\pi_*T}$.

We can now describe the tangent vectors to $\pi_*T$.
Recall that the measure $\sigma_T$ can be disintegrated along the fibers of the projection $\pi$ as follows.  If $f$ is a measurable function we have that
$$\int f(x) \sigma_T(x) = \int_L \left(\int_{\pi^{-1}(z)} f(x)  \sigma_T(x|\pi^{-1}(z))\right) (\pi_*\sigma_T)(z),$$ with the  usual notation $\sigma_T(\cdot|\pi^{-1}(z))$ for the conditional measure of $\sigma_T$ on the fiber.

If now $\varphi$ is a test $(p,p)$ form on $L$, we have
\begin{align*}
\bra{\pi_*T,\varphi}&= \bra{T,\pi^*\varphi} =
 \int \bra{t_T(x),(\pi^*\varphi)(x)} \sigma_T(x)\\
 &=\int_L\left(\int_{\pi^{-1}(z)}\bra{t_T(x),(\pi^*\varphi)(x)} \sigma_T(x|\pi^{-1}(z))\right) (\pi_*\sigma_T)(z)
 \\   &=\int_L\left(\int_{\pi^{-1}(z)}\bra{\pi_*(t_T(x)),\varphi(\pi(x))} \sigma_T(x|\pi^{-1}(z))\right) (\pi_*\sigma_T)(z)\\
 &=\int_L \bra{\widetilde{t}(z),\varphi(z)}(\pi_*\sigma_T)(z) \text{ where }  \widetilde{t}(z)=\int_{\pi^{-1}(z)}\pi_*(t_T(x))\sigma_T(x|\pi^{-1}(z))\\
 &=\int_L  \bra{h(z) \widetilde{t}(z),\varphi(z)} \sigma_{\pi_*T}(z).
\end{align*}
We see that the last integral is actually the  integral representation of $\pi_*T$, so  for $\sigma_{\pi_*T}$ a.e. $z$,
\begin{equation}\label{eq:average}
t_{\pi_*T}(z) =h(z) \widetilde{t}(z) =
h(z)\int_{\pi^{-1}(z)}\pi_*(t_T(x))\sigma_T(x|\pi^{-1}(z)).
\end{equation}
This implies  in particular that $\pi_*T$ is strongly positive, since $t_{\pi_*T}$ is a.s. an average of strongly positive $(p,p)$ vectors.

\medskip

We are now in position to conclude the proof of the theorem. We argue by contradiction, so let us assume that there exists  a set $A$ of positive trace mass  such that $t_T(x)$ has rank $\geq \ell$ for  $x\in A$. Let $S=T\rest{A}$, and consider the current $\pi_*S$
on $L$.
Then $\pi_*S$ satisfies the assumptions of Theorem \ref{thm:rank}, since it
 is  dominated by the positive closed current $\pi_*T$. Since $\dim(\sigma_{\pi_*S})<2\ell$,  we infer that $(\pi_*S)_{\rm ac}=0$, therefore $\rank(\pi_*S)<\ell$ a.e.
 
Now  by \eqref{eq:average}, for a.e. $z$, $t_{\pi_*S}(z)$ is an average of
 $\pi_*(t_S(x))$ with $x\in \pi^{-1}(z)$. Thus by Lemma \ref{lem:linalg} {\em i.} below, 
$\rank(\pi_*(t_S(x)))<\ell$ for $\sigma_S(\cdot|\pi^{-1}(z))$-a.e. $x$. On the other hand,  by  Lemma \ref{lem:linalg} {\em ii.}, if $I$ is chosen generically, $\rank(\pi_*(t_S(x)))\geq \ell$, $\sigma_S$-a.e.  This contradiction finishes the proof.
\end{proof}

\begin{lem} Let $V$ be a Hermitian complex vector space  with associated (1,1) form $\beta$.
\label{lem:linalg}
\begin{itemize}
\item[\it i.] Let $(t_\alpha)_{\alpha\in \mathcal{A}}$ be a measurable family of strongly
positive $(p,p)$ vectors of trace 1, and $\mu$ be a probability measure on $\mathcal{A}$.
Let $t= \int_\mathcal{A} t_\alpha d\mu(\alpha)$. If $\rank(t)<\dim V$, then
for a.e. $\alpha$, $\rank(t_\alpha)< \dim V$.
\item[\it ii.] Let $p<\ell$ and fix  a complex subspace $L$  of dimension $\ell$. If
$K$ is a supplementary subspace to $L$, we denote by
 $\pi_{K, L}$ be the projection onto $L$ with kernel $K$.
Let $t$ be a strongly positive $(p,p)$ vector of rank $r\geq \ell$.
Then there exists a set $\mathcal{E}(t)$ of zero Lebesgue measure in the corresponding Grassmannian
  such that, if $K\notin \mathcal{E}(t)$,  $\rank((\pi_{K,L})_*(t))= \ell$.
\end{itemize}
\end{lem}

\begin{proof}[Proof of Lemma \ref{lem:linalg}]

{\it i.} Recall that $\rank(t) = \rank(t\llcorner \beta^{p-1})$ so it is enough to prove the result for positive (1,1) vectors, that is, nonnegative Hermitian matrices. But in this context the result is obvious, as follows for instance from the concavity of
$M\mapsto (\mathrm{det}(M))^{1/k}$.

\medskip

{\it ii.}
We  use the following fact: if $p\leq \ell$ and $W$ is a $p$-dimensional subspace, then the set of $K$'s such that $\pi_{K,L}\rest{W}:W\cv L$ is injective is open and of full   measure.

Fix a decomposition $t=\sum_{k=1}^s t_k$ of $t$ as a sum of decomposable vectors.
Since $p<\ell$, by the previous observation we can assume that  for each $k$,
$\pi_{K,L}\rest{\mathrm{Span}(t_k)}$ is injective 
Furthermore, let us choose $\ell$ linearly independent vectors
$e_1, \ldots, e_\ell$ belonging to $\bigcup_{k=1}^s\mathrm{Span}(t_k)$. We may assume that
$\pi_{K,L}\rest{\mathrm{Vect}(e_1, \ldots, e_\ell)}$ is injective as well. Thus, $(\pi_{K,L})_*(t_k)$ is a non-trivial decomposable element of $\bigwedge^{p,p}(L)$, and by our second requirement
$\rank(\sum(\pi_{K,L})_*(t_k))\geq \ell$, whence the result.
\end{proof}

\begin{rmk}
It is clear from the proof that  a sharper condition for  $\mathrm{rank}(T)<\ell$ a.e. is that
for  a generic linear
 projection $\pi$ onto $\mathbb{P}^\ell$, $\pi_*\sigma_T$ is singular w.r.t. Lebesgue measure.
\end{rmk}


\begin{thebibliography}{[ABCD]}
\bibitem[De]{demailly} Demailly, Jean-Pierre.
{\em Complex analytic and differential geometry, Chap. III.} Book available online at 
{\tt http://www-fourier.ujf-grenoble.fr/$\sim$demailly/manuscripts/agbook.pdf}.
\bibitem[DS]{ds-pullback} Dinh, Tien-Cuong; Sibony, Nessim. {\em Pull-back currents by holomorphic maps.}  Manuscripta Math.  123  (2007),  357--371.
\bibitem[Du]{fatou} Dujardin, Romain {\em Fatou directions along the Julia set for endomorphisms of $\mathbb{CP}^k$.} J. Math. Pures Appl. to appear. 
\bibitem[Le]{lelong} Lelong, Pierre. {\em Fonctions plurisousharmoniques et formes différentielles positives.} Gordon \& Breach, 1968.
\end{thebibliography}
\end{document}